\documentclass[12pt,reqno]{article}

\usepackage[usenames]{color}
\usepackage{amssymb}
\usepackage{amsmath}
\usepackage{amsthm}
\usepackage{amsfonts}
\usepackage{amscd}
\usepackage{graphicx}

\usepackage[colorlinks=true,
linkcolor=webgreen,
filecolor=webbrown,
citecolor=webgreen]{hyperref}

\definecolor{webgreen}{rgb}{0,.5,0}
\definecolor{webbrown}{rgb}{.6,0,0}

\usepackage{color}
\usepackage{fullpage}
\usepackage{float}

\newcommand{\seqnum}[1]{\href{https://oeis.org/#1}{\rm \underline{#1}}}

\usepackage{amssymb}
\usepackage{graphicx}
\usepackage[margin=1in]{geometry}
\usepackage{amsthm}
\usepackage{mathtools}
\usepackage{tikz}
\usepackage[T1]{fontenc}
\usepackage{textcomp}
\usepackage{diagbox}

\graphicspath{{./figures/}}
\usepackage[colorinlistoftodos,prependcaption,textsize=tiny]{todonotes}
\usepackage[shortlabels]{enumitem}
\usepackage{mathrsfs}

\usepackage{color}
\usepackage{graphicx}
\usepackage{caption}
\usepackage{subcaption}

\begin{document}

\theoremstyle{plain}
\newtheorem{theorem}{Theorem}
\newtheorem{corollary}[theorem]{Corollary}
\newtheorem{lemma}[theorem]{Lemma}
\newtheorem{proposition}[theorem]{Proposition}

\theoremstyle{definition}
\newtheorem{definition}[theorem]{Definition}
\newtheorem{example}[theorem]{Example}
\newtheorem{conjecture}[theorem]{Conjecture}

\theoremstyle{remark}
\newtheorem{remark}[theorem]{Remark}

\begin{center}
\vskip 1cm{\LARGE\bf Analytical Study and Efficient Evaluation of the Josephus Function 
} \vskip 1cm \large
Yunier Bello-Cruz\footnote{The first author was partially supported by NSF Grant DMS--2307328 and the internal research and artistry (R\&A) grant at Northern Illinois University.} and Roy Quintero-Contreras\footnote{The second author wishes to thank the Shippensburg University of Pennsylvania for its online public source ``The Josephus Game'', which was used hundreds of times to check out values of the Josephus function along the first stages of this research, and wants to dedicate this paper to his \textit{alma mater} ``Universidad Central de Venezuela'' (UCV) for its $300$th anniversary of creation.}\\
Department of Mathematical Sciences\\
Northern Illinois University\\
DeKalb, IL   60115 \\
USA\\
\href{mailto:yunierbello@niu.edu}{\tt yunierbello@niu.edu}\\
\href{mailto:rquinterocontreras@niu.edu}{\tt rquinterocontreras@niu.edu}

\end{center}
\vskip .2 in

\begin{abstract}
A new approach to analyzing intrinsic properties of the Josephus function, $J_{_k}$, is presented in this paper. The linear structure between extremal points of $J_{_k}$ is fully revealed, leading to the design of an efficient algorithm for evaluating $J_{_k}$. Algebraic expressions that describe how recursively compute extremal points, including fixed points, are derived. The existence of consecutive extremal and also fixed points for all $k\geq 2$ is proven as a consequence, which generalizes Knuth's result for $k=2$. Moreover, an extensive comparative numerical experiment is conducted to illustrate the performance of the proposed algorithm for evaluating the Josephus function compared to established algorithms. The results show that the proposed scheme is highly effective in computing $J_{_k}(n)$ for large inputs.
\end{abstract}

\section{Introduction}

The Josephus problem is a game of elimination that has been studied for nearly two millennia. The earliest known formulation of the problem appears in the historical text \cite{Jos} written by historian Flavius Josephus. Josephus described a method of serial elimination by casting lots, which he and $40$ of his soldiers applied while trapped in a cave by the Roman army and facing imminent capture and inevitable massacre.  

The general formulation of the problem is as follows: a certain number of people $n$ is arranged in a circle, and an execution method is established. An initial counting point and a direction of rotation are fixed, and after $k-1$ number of people is counted, the next person is executed and removed from the circle. This procedure is repeated with the remaining people until only one person remains, who is released. The objective of this problem is to determine the position of the survivor $J_{_k}(n)\in \{1,2,\ldots,n\}$ in the initial circle. 

The Josephus problem received little attention throughout the first millennium, but new formulations were proposed in the second millennium, such as the rhyming games traced by Oring \cite{Ori}. In the following centuries, more specific versions of the problem emerged, and it was integrated into a certain category of puzzles or riddles. One notable example is Bachet's presentation \cite{Bac}, where was established a formal procedure to solve the case $k=3$ and $n=41$ with total accuracy. In $1776$, Euler heuristically found a recursive formula that connects the survivor's position for a given number of condemned with the position of the corresponding survivor when one more condemned is added; see \cite{Eul}. This was a significant development in the mathematical treatment of the problem, as it required logical argumentation, appropriate symbolic notation, and exploration. However, the computational aspect of the problem has received more attention in recent years, as the advancement of technology allows for the efficient calculation of the survivor's position $J_{_k}(n)$ for a large number of people $n$ and a fixed $k$. In \cite{Tai}, Tait presents some discussion and figures in this direction, highlighting the importance of fast computational methods for solving the problem. Over time, the problem has attracted the attention of mathematicians from various areas, and the study of the problem was recently extended to the field of permutations. Mathematicians such as Herstein, Kaplansky, Ball, and others studied the problem in group theory and algorithmically; see, for instance, \cite{Dow, Cha, Her, Rou, Arh, Wil}. The Josephus problem also has various modern applications such as computer algorithms, data structures, and image encryption; see, for instance, \cite{Woo,Cos,Nai,Yan}.  When the number of jumps, $k-1$, is equal to one, Knuth derived a closed-form expression for $J_{_2}(n)$ as $2n - 2^{\lfloor \log_2{n} \rfloor + 1} + 1$ and developed an efficient algorithm for evaluating $J_{_k}$ for arbitrary values of $k$; see, for instance, \cite{Knuth}. This algorithm avoids the recursive nature of the Josephus function and is presented in \cite{Gra}. We invite the reader interested in delving deeper into the historical origin of this problem and applications to review references such as \cite{Odl, Hal, Rob, Rec, Smi, Pet, New} and \cite[Appendix]{Gro}.

This paper introduces a new approach for analyzing the intrinsic properties of the discrete Josephus function, $J_{_k}$. We formulate algebraic expressions that describe and characterize all extremal points of $J_{_k}$ and include recurrence formulas to compute low and high extremal points. Additionally, we prove the existence of consecutive extremal and fixed points for all $k\geq 3$, generalizing Knuth's result for $k=2$; see \cite{Gra}. Moreover, by revealing the discrete and piecewise linear structure of $J_{_k}$ between extremal points, we design an efficient algorithm for evaluating $J_{_k}(n)$ for large $n$ and a given reduction constant $k$. A comparative computational study is conducted at the last section of the paper to evaluate the performance of the proposed algorithm against established methods, such as Euler/Woodhouse \cite{Eul, Woo}, Knuth \cite{Gra}, and Uchiyama \cite{Uch}. The results of the numerical comparison indicate that the proposed scheme is highly effective for computing $J_{_k}(n)$ for large inputs $k$ and $n$.

\subsection{Notation and Definitions} 

The mathematical formulation of the classical Josephus problem can be stated as follows: Let $n$ be the number of people arranged in a circle which closes up its ranks as individuals are picked out. Starting anywhere (person 1st spot), go sequentially around clockwise, picking out each $k$th person (this number $k$ is called the \textit{reduction constant}) until but one person is left (this person is called the \textit{survivor}). The position of the survivor is denoted by $J_{_k}(n)$, which belongs to the natural numbers $\mathbb{N}$. This procedure is called the \textit{elimination process}, and it naturally generates a discrete function $J_{_k}: \mathbb{N} \rightarrow \mathbb{N}$ for each $k \geq 2$ that we will call the Josephus function. Given $n \geq 1$ and $k \geq 2$ integers, we say that the Josephus problem has been solved once we have determined the value of $J_{_k}$ at $n$. For any two integers $\ell$ and $m$ such that $\ell\le m$, the set $\{\ell,\ldots,m\}$ will be denoted by $[[\ell, m]]$. Notice further that $J_{_k}(n) \in [[1,n]]$ for every $n$.

\begin{definition}[Fixed and Extremal points]
A \textit{fixed point} of $J_{_k}$ is a value $n_{_p}$ such that $n_{_p}=J_{_k}(n_{_p})$. Additionally, an \textit{extremal point} $n_{_e}$ is defined as a point that satisfies either $J_{_k}(n_{_e})\in [[1,k-1]]$ or $J_{_k}(n_{_e})\in [[n_{_e}-k+2,n_{_e}]]$. If $J_{_k}(n_{_e})\in [[1,k-1]]$, we refer to $n_{_e}$ as a \textit{low extremal point}. On the other hand, if $J_{_k}(n_{_e}) \in [[n_{_e}-k+2,n_{_e}]]$, we refer to $n_{_e}$ as a \textit{high extremal point}. 
\end{definition}
Note that, a fixed point $n_{_p}$ is also a high extremal point because $J_{_{k}}(n_{_p})=n_{_p}$. However, there are high extremal points that are not fixed points; see Figure \ref{f1} below.

\section{Properties of the Josephus Function} 
In this section, we begin by recalling a recursive formula, which first appeared in Euler's paper \cite[\& 8, pp. 130-131]{Eul} and establishes a way of determining $J_{_{k}}(n+1)$ in terms of $J_{_k}(n)$.

\begin{theorem}[Euler's formula]\label{T1}
Let $k \geq 2$ and denote $p := J_{_k}(n)$. Then, $J_{_k}(n+1)=p+k-\ell(n+1)$, if $p+k \in [[\ell(n+1)+1, (\ell+1)(n+1)]]$ for some non-negative integer $\ell$.
\end{theorem}
\begin{figure}[!h]
\centering
\begin{subfigure}[b]{0.49\textwidth}
\includegraphics[width=1\textwidth]{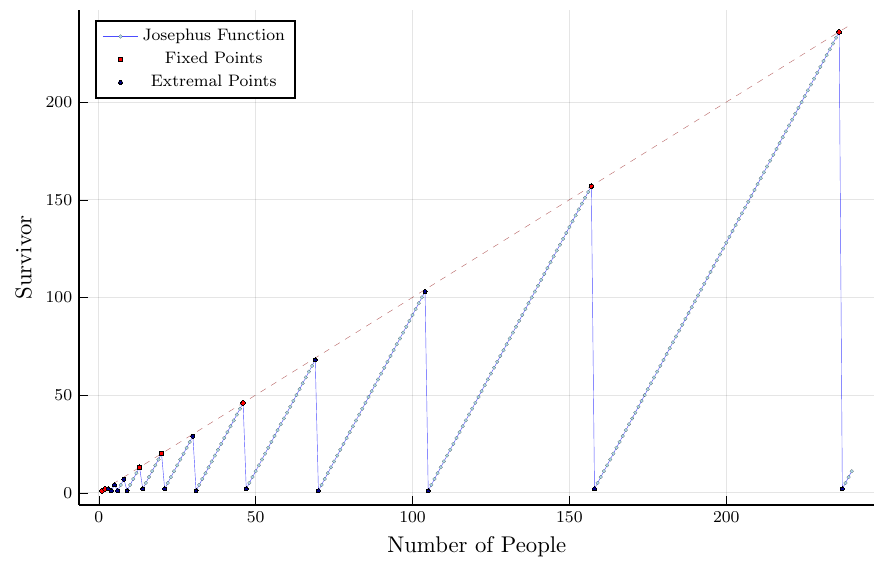}
\caption{Graph of $J_{_3}(n)$ for $n\le 240$}
\label{fig:1.1}
\end{subfigure}
\hfill
\begin{subfigure}[b]{0.49\textwidth}
\includegraphics[width=1\textwidth]{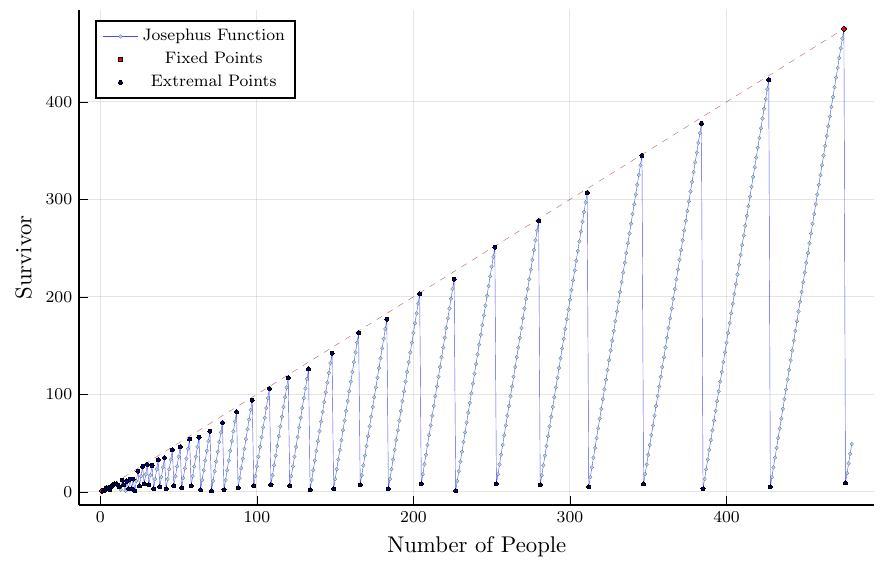}
\caption{Graph of $J_{_{10}}(n)$ for $n\le 480$}
\label{fig:1.2}
\end{subfigure}
\hfill
\caption{Graphs of the Josephus functions $J_{_3}$ and $J_{_{10}}$}
\label{f1}
\end{figure}

\begin{remark}\label{R1} For some special values of the reduction constant $k$, we have the following observations: 
\item [ {\bf (a)}] If $k\le n+1$, then the formula given in Theorem \ref{T1} can be simplified as follows:
\begin{equation}\label{E2}J_{_k}(n+1)=\begin{cases}p+k, & \text{ if \ \ \ } p+k \leq n+1\\
p+k-(n+1),  & \text{ if \ \ \ } p+k > n+1.\end{cases}
\end{equation}
\item [ {\bf (b)}] When $k=3$, \eqref{E2} holds for every $n$. In fact, we just need to check \eqref{E2} for $n=1$. In this case, we have
$J_{_3}(n+1)=J_{_3}(2)=2= J_{_3}(1)+3-(1+1)=J_{_3}(n)+k-(n+1)$.
\item [ {\bf (c)}] When $k=2$, no recursive formula is necessary hereafter since Knuth \cite{Gra} has deduced the following explicit formula:
$J_{_2}(n)=2n-2^{\lfloor \log_2n \rfloor+1}+1,$ for all $n$.  
\item [ {\bf (d)}] It is possible for two consecutive values of $n$ to have the same image under $J_{_k}$. For example, when $k=8$, $J_{_8}(3)=J_{_8}(4)=3$.
\item [ {\bf (e)}] If $n\le k-1$, $J_{_k}(n)\le n\le k-1$, which implies that $n$ is a low, and also high, extremal point. Moreover, if $n\ge 2k-3$, then there is an alternating sequence of pure low and high extremal points for the Josephus function $J_{_k}$ for every $k\ge 2$. 
\end{remark}

In Figure \ref{f1}, we present the graphs of $J_{_3}$ and $J_{_{10}}$ for various values of $n$. The extremal points of $J_{_k}$, $n_{_e}$, \emph{i.e.}, $J_{_k}(n_{_e})\in [[1,k-1]]\cup [[n_{_e}-k+2,n_{_e}]]$, which will be fully studied in the next subsection, are indicated in blue. Moreover, the fixed points $n_{_p}$ of $J_{_k}$, \emph{i.e.}, $n_{_p}=J_{_k}(n_{_p})$, that are also extremal points, are depicted in red. 

Note that for $n_{_e}$ to be both a low and high extremal point, $J_{_{k}}(n_{_e})\in [[1,k-1]]\cap [[n_{_e}-k+2,n_{_e}]]$, which always happens if $n_{_e}\leq k-1$. In addition, if $n_{_e}\geq 2k-3$ is a high extremal point then $n_{_e}+1$ is a low extremal point, which reveals a sequence of alternated extremal points.  The two graphs depicted in Figure \ref{f1} exhibit the intrinsic piecewise linear structure of $J_{_k}$ between its low and high extremal points. Moreover, in general the fixed points for $k\ge 3$ exhibit a chaotic behavior in contrast with the case $k=2$; see Figure \ref{f1} above for the cases $k=3$ and $k=10$, and also Table \ref{t1} below with $k=2, 3,\ldots , 12, 15$. Note that the fixed points of $J_{_2}$ can be described by the formula: $n_{_p}^{(i)}=2^{i}-1$, where $i$ is a natural number. Note further that $J_{_2}(2^i)=1$, which makes $n_{_e}=2^i$ be a low extremal point, for every $i$; see, for instance,  \cite[pp. 184]{Knuth}. 

In Table \ref{t1} below, we illustrate the first fixed points of the Josephus functions $J_{_k}$ for some values of $k$. 
\begin{table}[h!]
\centering
\small
\scalebox{0.8}{
\begin{tabular}{|| c | c | c | c | c | c | c | c | c | c | c | c | c ||}
 \hline
  \diagbox{Fixed points}{Functions} & $J_{_2}$ & $J_{_3}$ & $J_{_4}$ & $J_{_5}$ & $J_{_6}$ & $J_{_7}$ & $J_{_8}$ & $J_{_9}$ & $J_{_{10}}$ & $J_{_{11}}$ & $J_{_{12}}$ & $J_{_{15}}$\\
  \hline 
  1st  & 1 & 1 & 1 & 1 & 1 & 1 & 1 & 1 &  1 & 1 & 1 & 1\\
  \hline 
  2nd  & 3 & 2 & 21 & 2 & 20 & 2 & 3 & 2 & 4 & 2  & 10 & 2\\
  \hline 
  3rd  & 7 & 13 & 38 & 46 & 51 & 3 & 13 & 7 & 475 & 4  & 11 & 52\\ 
  \hline
  4th  & 15 & 20 & 51 & 542 & 794 & 12 & 15 & 8 & 8177 & 5  & 19 & 388\\ 
  \hline
  5th  & 31 & 46 & 122 & 2587 & 953 & 68 & 26 & 15 & 11217 & 49  & 55 & 1899\\ 
   \hline
  6th  & 63 & 157 & 163 & 3234 & 17629 & 274 & 1276 & 17 & 28954 & 54  & 111 & 30003\\
  \hline 
  7th  & 127 & 236 & 689 & 6317 & 21155 & 593 & 1905 & 375 & 126567 & 188  & 290 & 136887\\ 
  \hline
\end{tabular}}
\caption{First seventh fixed points of $J_{_k}$ for $k=2,3,\ldots,12,15$}\label{t1}
\end{table}

Now based on Theorem \ref{T1}, we present a formula for the Josephus function for the value of $n=n_{_p}+1$, \emph{i.e.}, just after a fixed point $n_{_p}=J_{_k}(n_{_p})$ is attained. 

\begin{proposition}[Values of $J_{_k}$ after reaching a fixed point]\label{P1}
 Let $k \geq 2$ and $n_{_p}$ be a fixed point. Then,
\begin{equation*}J_{_k}(n_{_p}+1)=\begin{cases}k-1, & \text{ if \ \ } k \leq n_{_p}+1\\
n_{_p}+1,  & \text{ if \ \ } k = s_{_\ell}+1 \;  \text{ for some  } \ell\geq 1 \\
k_{_1}-1, & \text{ if \ \ } k = s_{_\ell}+k_{_1} \; \text{ for some } \ell\geq 1 \text{ and } 2 \leq k_{_1} \leq n_{_p} \\ 
n_{_p}, & \text{ if \ \ } k = s_{_{\ell+1}} \; \text{ for some } \ell\geq 1,
\end{cases}\end{equation*}  where $s_{_\ell}=\ell(n_{_p}+1)$. Moreover,  $J_{_k}(n_{_p})=n_{_p}$ if and only if $J_{_k}(n_{_p}+1)=k-1$ whenever $n_{_p}\ge k-1$. \end{proposition}
\begin{proof}
Let us consider the four cases for the values of $k$ with respect to $n_{_p}$ as follows:

\noindent {Case 1.} Assume $k \leq  n_{_p} + 1$. Note that $s_{_1} + 1  = n_{_p} + 2  \leq n_{_p} + k  \leq  2n_{_p} + 1  <  s_{_2}$. It follows from Theorem \ref{T1} that $J_{_k}(n_{_p} + 1)  =  n_{_p} + k - s_{_1}  =  n_{_p} + k - (n_{_p} + 1)  =  k - 1$.
 
\noindent {Case 2.} Assume  $k =  s_{_\ell} + 1$. So, $n_{_p} + k  = s_{_{\ell+1}}$.
Then, by Theorem \ref{T1}, $J_{_k}(n_{_p} + 1)  = n_{_p} + k - s_{_\ell}  = n_{_p} + 1$.

\noindent {Case 3.} Assume  $k =  s_{_\ell} + k_{_1}$ ($2 \leq k_{_1} \leq n_{_p}$). Note that $s_{_{\ell+1}} + 1 \leq  n_{_p} + k \leq 2n_{_p} + s_{_\ell} =  s_{_{\ell+1}} + n_{_p} - 1  <  s_{_{\ell+2}}$. Then, by Theorem \ref{T1}, $J_{_k}(n_{_p} + 1)  = n_{_p} + k - s_{_{\ell+1}} = k_{_1} - 1$.

\noindent {Case 4.} Assume  $k = s_{_{\ell+1}}$. Then, $s_{_{\ell+1}} + 1 \leq n_{_p} + k < s_{_{\ell+2}}$. Theorem \ref{T1} implies $J_{_k}(n_{_p} + 1) = n_{_p} + k - s_{_{\ell+1}} = n_{_p}$.

On the other hand, assume that $k \leq n_{_p} + 1$ and $J_{_k}(n_{_p} + 1) = k - 1$. By using Euler's formula backwards, we can determine the value of $p = J_{_k}(n_{_p})$. In particular, $p$ must satisfy the equation $k - 1 = p + k - s_{_\ell}$ for some non-negative integer $\ell$. This is equivalent to $p = s_{_\ell} - 1$. If $\ell = 0$, then $p = -1$, which is not possible. If $\ell \geq 2$, then $p \geq s_{_2} - 1 > s_{_1} - 1 = n_{_p}$, which is also not possible. Hence, $\ell = 1$ and $p = n_{_p}$, which implies that $n_{_p}$ is a fixed point of $J_{_k}$.
\end{proof}

\begin{remark}\label{remark-2}
If $n_{_p}$ is a fixed point, then the following interesting facts can be derived from Proposition \ref{P1}:
\item [ {\bf (a)}] If $n_{_p} \geq k-1$, then $J_{_k}(n_{_p}+1)=k-1$. In particular, $J_{_3}(n_{_p}+1)=2$ for any $n_{_p}$.
\item [ {\bf (b)}] It is possible for two consecutive fixed points to exist for a given value of $k$. For example, when $k=9$, $7$ and $8$ are both fixed points of $J_{_9}$.

\item [ {\bf (c)}] Note that, if a fixed point $n_{_p}\le k-1$, then $n_{_p}$ is also a low extremal point. 
\end{remark}
    
We will now provide a formal presentation of a partially generalized version of Euler's formula, which was previously stated without proof in \cite[pp. 47]{Rob}.

\begin{lemma}[Generalization of the Euler's formula]\label{L1}
Let $n \geq 1$ and $k \geq 2$ be positive integers and denote $p := J_{_k}(n)$. Then, 

\item [ {\bf (a)}] If $p+km \leq n+m$, for any $m\in\mathbb{N}$, then $J_{_k}(n+m)=p+km$.

\item[ {\bf (b)}] If $m\in\mathbb{N}$ is the smallest value for which $p+km > n+m$, then $J_{_k}(n+m)=p+km-(n+m).$
\end{lemma}

\begin{proof} \noindent {\bf (a)} Let $m$ be any natural number. First of all, $p+km \leq n+m$, if and only if, $m\leq (n-p)/(k-1)$. Thus, $p+k\check{m} \leq n+\check{m}$ for every $0\leq \check{m} \leq m$. Let us define the finite sequences: $r_{_{\check{m}}}:=n+\check{m}$, and $t_{_{\check{m}}}:= J_{_k}(r_{_{\check{m}}})$ for every integer $\check{m}=0,\ldots,m$. Since, $p+k\leq n+1$, by Theorem \ref{T1}, we get $t_{_1}=p+k$. Now, since $t_{_1}+k=p+k(2) \leq n+2=r_{_1}+1$, again by Theorem \ref{T1}, we get $t_{_2}=t_{_1}+k$. By applying this procedure $m$ times, we get that $t_{_{\check{m}}}=t_{_{\check{m}-1}}+k$ for $\check{m}\in[[1,m]]$. Then,
$$J_{_k}(n+m)= J_{_k}(r_{_m})=t_{_m}=t_{_{m-1}}+k=\cdots=p+km,$$ proving item (a). 
   
\noindent {\bf (b)} Let us denote $p':=p+k(m-1)$ and $n':=n+(m-1)$. Note that $p'\leq n'$ by hypothesis. Thus, by part (a), $J_{_k}(n')=p'$. Moreover $p'+k>n'+1$. We claim that $p'+k\leq 2(n'+1)$. Otherwise, $2(n+m)=2(n'+1) < p'+k \leq n'+k =(n+m)+k-1$, which implies that $n+m<k-1$. Then, $k<p+k(m-1)\leq (n+m)-1<k-2$, which is a contradiction. Hence, it follows from equation (\ref{E2}) that $J_{_k}(n+m)=J_{_k}(n'+1)=p'+k-(n'+1)=p+k-(n+m),$ which proves item (b). 
\end{proof}

In the following subsection, we will use Lemma \ref{L1} to derive several general conclusions regarding the behavior of the Josephus function $J_{_k}$. 
\subsection{Characterization of extremal points}\label{char-extremal}

We start by characterizing some initial extremal points for the Josephus function.

\begin{corollary}[The first known high extremal point]\label{coro2.4}
    The point $2k-2$ is always an extremal point for $J_{_k}$. In particular, \begin{equation}\label{def-ne}
    n_{_e}:= \begin{cases} 2k-3, & \text{ if } J_{_k}(2k-2) \leq k-1  \\
    2k-2, & \text{ if } J_{_k}(2k-2) > k-1,\end{cases}\end{equation} is a high extremal point and $n_{_e}+1$ is a low extremal point.
\end{corollary}

\begin{proof}
When $J_{_k}(2k-2) \leq k-1$, by the definition of low extremal point, we have that $2k-2$ is a low extremal point and hence $n_{_e}=2k-3$ is a high extremal point proving the first part of equation \eqref{def-ne}.

On the other hand, when $J_{_k}(2k-2)>k-1$ we have
$$J_{_k}(2k-2)\geq k =(2k-2) -k+2 \in [[(2k-2) -k+2,2k-2]].$$ Hence, $n_{_e}=2k-2$ is a high extremal point, as desired. The rest of the statement of the corollary follows from the definition of low extremal points. 
\end{proof}

In the following, we present our main results. From now on, let us assume that $k\ge 2$ and $n\ge k-1$, which allows an alternating sequence of low and high extremal points for the Josephus function $J_{_k}$ where $k\ge 2$. We next reveal a recurrence procedure to find successive extremal points of the function $J_{_k}$. In particular, we prove that there is an infinite number of fixed points for $J_{_k}$.  
 
First, we introduce the following sequences: 
\begin{align}\label{delta}\{\delta_{_k}(r,j)\}_{r=0, j=1}^{k-2,k-1}\quad \mbox{with}\quad \quad\delta_{_k}(r,j)&:=\begin{cases}  1,  &\text{ if } r < j-1 \\                                                                0, & \text{ if } r \geq j-1, \\
                             \end{cases}\\ \label{ak}\{a_{_{k}}(n,r)\}_{n=k-1,r=0}^{\infty, k-2}\quad \mbox{with}\quad\quad a_{_{k}}(n,r)&:=\dfrac{k(n+1)-(r+1)}{k-1}, \\  \label{mk} \{M_{_{k}}(n,r,j)\}_{n=k-1, r=0, j=1}^{\infty, k-2, k-1}\quad \mbox{with}\quad M_{_{k}}(n,r,j)&:=\dfrac{n-r}{k-1}-\delta_{_k}(r,j). 
\end{align}
The meaning of each sequence term can be summarized as follows: $\delta_{_k}(r,j)$ serves as a variation of Kronecker's delta, while $a_{_k}(n,r)$ is associated with the extremal points of the Josephus function. Lastly, $M_{_k}(n,r,j)$ defines the upper bound of an interval where the Josephus function exhibits linear behavior. We outline some important properties of these sequences in the following lemma.

\begin{lemma}[Properties of the Josephus function]\label{L 2} Denote $j:=J_{_k}(n+1)$. Then,

\item [ {\bf (a)}] If $j\in [[1,k-1]]$, i.e., $n+1$ is a low extremal point, then $m_{_0} := \left \lfloor \dfrac{n-(j-1)}{k-1} \right \rfloor$ is the largest nonnegative integer $m$ that satisfies the inequality
\begin{equation}\label{Eq A}
j+km \leq n+1+m.
\end{equation}
and $[[0, m_{_0}]]$ is the solution set of \eqref{Eq A}. Moreover, $m_0\ge 1$ if $n\ge k-2+j$.
\item [ {\bf (b)}] If $j\in [[1,k-1]]$, i.e., $n+1$ is a low extremal point, and $n\equiv r \mod (k-1)$, then $m_{_0}= M_{_k}(n,r,j)$ and $n+1+m_{_0}=a_{_k}(n,r)-\delta_{_k}(r,j)$.
\item [ {\bf (c)}] $a_{_k}(n,r)$ is an integer, if and only if, $n \equiv r \mod (k-1)$.
\end{lemma}
\begin{proof} \noindent {\bf (a)}  First note that \eqref{Eq A} is equivalent to $m\le (n-(j-1))/(k-1)$. Then, since $n\ge k-1$, we have $m_0=\left\lfloor(n-(j-1))/(k-1)\right\rfloor$ as the largest non-negative integer satisfying \eqref{Eq A}. Therefore, any non-negative integer $m$ satisfying $0\le m\le m_0$ clearly satisfies the inequality \eqref{Eq A}, since $0$ is a trivial solution. Now, note that $n\ge k-2+j$ happens automatically when $n\ge 2k-3$ since $j\le k-1$. Moreover, since $(n-(j-1))/(k-1)\ge (2k-3 - (k-2))/(k-1)=1$, we have $m_0\ge 1$.

On the other hand, if $n\in [[k-1, 2k-4]]$ then, from a direct inspection to the value of $(n-(j-1))/(k-1)$, we get that
\begin{equation}\label{m0-n-small}m_{_0}=\begin{cases} 0, & \text{ if  } n< k-2+j \\
1,  & \text{ if  } n\ge k-2+j. \end{cases}\end{equation} So, in this case, it follows from \eqref{m0-n-small} that $m_0=1$ is the largest integer satisfying \eqref{Eq A} if $n\ge k-2+j$, proving this item. 

\noindent {\bf (b)} Suppose that $n \equiv r \mod (k-1)$. Note that $\frac{(j-1)-r}{k-1} \in [\frac{1}{k-1}, \frac{j-1}{k-1}] \subset (0,1)$ when $r<j-1$, and $\frac{r-(j-1)}{k-1} \in [\frac{1}{k-1}, 1-\frac{j}{k-1}] \subset (0,1)$ when $r>j-1$. Thus, using the definition of $M_{_k}(n,r,j)$ given in \eqref{mk}, we get 
\begin{align*} m_{_0}  &= \left.\begin{cases} \ \left \lfloor \dfrac{n-r}{k-1} -\dfrac{(j-1)-r}{k-1} \right \rfloor = \dfrac{n-r}{k-1} -1,& \text{ if } r < j-1 \\ \ \dfrac{n-r}{k-1},     & \text{ if } r = j-1 \\
\ \left \lfloor \dfrac{n-r}{k-1} +\dfrac{r-(j-1)}{k-1}\right \rfloor = \dfrac{n-r}{k-1}, & \text{ if } r > j-1
\end{cases}\right \}= M_{_k}(n,r,j).
\end{align*}
The equality $n+1+m_{_0}=a_{_k}(n,r)-\delta_{_k}(r,j)$ follows directly from previous equation and by the definition of $a_{_k}(n,r)$ and $\delta_{_k}(r,j)$ in \eqref{delta} and \eqref{ak}, respectively.

\noindent {\bf (c)} If $a_{_k}(n,r)=\dfrac{k(n+1)-(r+1)}{k-1}$ is an integer, clearly $k(n+1)\equiv r+1 \mod(k-1)$. 
In this case,
\begin{equation*}
n \equiv kn-(k-1)n \equiv kn \equiv k(n+1)-k \equiv r+1-k \equiv r-(k-1) \equiv r \mod(k-1).
\end{equation*}
Now if $n \equiv r \mod (k-1)$ then $k(n+1) \equiv n+1 + (k-1)(n+1) \equiv n+1 \equiv r+1 \mod (k-1)
$, which implies that $a_{_k}(n,r)$ is an integer.
\end{proof}

We now present formulas characterizing extremal points of the Josephus function $J_{_k}$ and their images. These expressions will be useful in determining recurrence formulas between extremal points and a formula for the image of $J_{_k}$ at any arbitrary $n$ in terms of the high and low extremal points.

\begin{theorem}[Extremal point formulas]\label{L 2-part2} Let $n_{_e}\geq 2k-3$ be a high extremal point and denote $j:=J_{_k}(n_{_e}+1)$. Then, 

\item [ {\bf (a)}] $J_{_k}(n_{_e}+1+m)=j+km$, for all $m\in \left[\left[0, \left \lfloor (n_{_e}-(j-1))/(k-1) \right \rfloor\right]\right]$.

\item [ {\bf (b)}] If $n_{_e} \equiv r \mod{(k-1)}$, then the next high extremal point of $J_{_k}$ is \begin{equation}\label{ne+hep}n_{_e}^{_{\rm +}}:=\dfrac{k(n_{_e}+1)-(r+1)}{k-1}-\delta_{_k}(r,j).\end{equation} Moreover, 
\begin{equation}\label{jk-ne+}J_{_k}(n_{_e}^{_{\rm +}})=n_{_e}^{_{\rm +}}-\delta_{_k}(r,j)(k-1)-r+(j-1),\end{equation}
\begin{equation}\label{jk-ne++1}J_{_k}(n_{_e}^{_{\rm +}}+1)=(1-\delta_{_k}(r,j))(k-1)-r+(j-1),\end{equation}
and 
\begin{equation}\label{jk-ne++1-new}J_{_k}(n_{_e}^{_{\rm +}}+1)=(k-1)-[n_{_e}^{_{\rm +}}-J_{_k}(n_{_e}^{_{\rm +}})].\end{equation}
\item [ {\bf (c)}] $n_{_e}^{_{\rm +}}$ is a fixed point of $J_{_k}$, if and only if, $r-(j-1)=0$ and $n_{_e}^{_{\rm +}}=\dfrac{k(n_{_e}+1)-j}{k-1}$.
\item [ {\bf (d)}] $J_{_k}$ is a linear function on $[[n_{_e}+1,n_{_e}^{_{\rm +}}]]$.
\end{theorem}
\begin{proof} 

\noindent {\bf (a)} By Lemma \ref{L 2}(a), $j+km \leq n_{_e}+1+m$ for every $m\in [[0,m_{_0}]]$ where $m_{_0}:= \left \lfloor (n_{_e}-(j-1))/(k-1) \right \rfloor$. Additionally, Lemma \ref{L1}(a) implies that $J_{_k}(n_{_e}+1+m)=j+km$ for every $m\in [[0,\left \lfloor (n_{_e}-(j-1))/(k-1) \right \rfloor]]$ as desired.

\noindent {\bf (b)}  Equation \eqref{ne+hep} follows directly from the previous item, and Lemma \ref{L 2}(b) that $m_0\ge 1$ and 
\begin{align*}
J_{_k}(n_{_e}^{_{\rm +}})&=J_{_k}(n_{_e}+1+m_{_0})=j+km_{_0}\\
&=\begin{cases} j+k\left(\dfrac{n_{_e}-r}{k-1}-1\right), & \text{ if } r < j-1 \\
j+k\left(\dfrac{n_{_e}-r}{k-1}\right), & \text{ if } r \geq j-1
\end{cases} \\
&=\begin{cases}a_{_k}(n_{_e},r)-1-(k-1)-(r-(j-1)), & \text{ if } r < j-1 \\
a_{_k}(n_{_e},r)-(r-(j-1)), & \text{ if } r \geq j-1\end{cases}\\
&= n_{_e}^{_{\rm +}}-\delta_{_k}(r,j)(k-1)-r+j-1,
\end{align*}
where $m_{_0}= \left \lfloor (n_{_e}-(j-1))/(k-1) \right \rfloor$, which proves \eqref{jk-ne+}.
Now using Lemma \ref{L1}(b) and Lemma \ref{L 2}(b), we have
\begin{align*}
J_{_k}(n_{_e}^{_{\rm +}}+1)&=J_{_k}(n_{_e}+1+(m_{_0}+1))=j+k(m_{_0}+1)-(n_{_e}+1+(m_{_0}+1))\\
    &=\begin{cases} j+k\left(\dfrac{n_{_e}-r}{k-1}\right)-n_{_e}-1-\dfrac{n_{_e}-r}{k-1}, & \text{ if } r < j-1 \\
         j+k\left(\dfrac{n_{_e}-r}{k-1}+1\right)-n_{_e}-2-\dfrac{n_{_e}-r}{k-1}, & \text{ if } r \geq j-1
         \end{cases} \\
         &=\begin{cases}-(r-(j-1)), & \text{ if } r < j-1 \\
         (k-1)-(r-(j-1)), & \text{ if } r \geq j-1\end{cases}\\
         &= (1-\delta_{_k}(r,j))(k-1)-r+j-1,
\end{align*} establishing \eqref{jk-ne++1}. On the other hand,
since $(1-\delta_{_k}(r,j))(k-1)-r+j-1 \in [[1,k-1]]$, $n_{_e}^{_{\rm +}}$ is certainly the next high extremal point of $J_{_k}$. Additionally,
\begin{align*}
(k-1)-[n_{_e}^{_{\rm +}}-J_{_k}(n_{_e}^{_{\rm +}})]&=(k-1)-[n_{_e}^{_{\rm +}}-(n_{_e}^{_{\rm +}}-\delta_{_k}(r,j)(k-1)-r+j-1)]  \\
      &=(k-1)-[\delta_{_k}(r,j)(k-1)+r-j+1] \\ &=(1-\delta_{_k}(r,j))(k-1)-r+j-1 \\
      &=J_{_k}(n_{_e}^{_{\rm +}}+1),
\end{align*} proving \eqref{jk-ne++1-new}.

\noindent {\bf (c)} If $n_{_e}^{_{\rm +}}$ is a fixed point of $J_{_k}$, then \eqref{jk-ne+} implies that $\delta_{_k}(r,j)(k-1)+r-j+1=0$. Therefore, $r-j+1=0$ and $n_{_e}^{_{\rm +}}=a_{_k}(n_{_e},j-1)=(k(n_{_e}+1)-j)/(k-1)$. Conversely, if $r-j+1=0$ and $n_{_e}^{_{\rm +}}=(k(n_{_e}+1)-j)/(k-1)$, then \eqref{jk-ne+} implies $J_{_k}(n_{_e}^{_{\rm +}})=n_{_e}^{_{\rm +}}-\delta_{_k}(r,j)(k-1)-r+j-1=n_{_e}^{_{\rm +}}-0(k-1)-0=n_{_e}^{_{\rm +}}$. Thus, $n_{_e}^{_{\rm +}}$ is a fixed point of $J_{_k}$.

\noindent {\bf (d)}  It follows from (a).
\end{proof}

The following results present a recurrence formula that allows us to compute successive high extremal points for the Josephus function $J_{_k}$. We also provide algebraic expressions for the images of $J_{_k}$ at extremal points and at an arbitrary $n$. This recurrence relation is of great importance in establishing the existence of fixed points and developing an effective extremal algorithm for evaluating the Josephus function.
\begin{corollary}[Recurrence formula for computing high extremal points]\label{L 3} Let $n_{_{e}}^{(i)}\ge 2k-3$ be a high extremal point of the Josephus function $J_{_k}$, with corresponding functional value $J_{_k}(n_{_e}^{(i)})$. Given 
\begin{equation}\label{ri}
r_{_{i}}:={\rm mod}(n_{_e}^{(i)},k-1),
\end{equation} and 
\begin{equation}\label{ci}
 c_{_{i}} := \begin{cases} 1, & \text{ if } r_{_{i}} < k-2-n_{_e}^{(i)}+J_{_k}(n_{_e}^{(i)})\\
0, & \text{ if } r_{_{i}} \geq k-2-n_{_e}^{(i)}+J_{_k}(n_{_e}^{(i)}).\end{cases}
\end{equation}
Then,  we can compute the next high extremal point 
\begin{equation}\label{nei+1}
n_{_{e}}^{(i+1)} =  \dfrac{k(n_{_e}^{(i)}+1)-(r_{_i}+1)}{k-1}-c_{_i},
\end{equation} and its functional value 
\begin{equation}\label{jk-nei+1}
J_{_k}(n_{_{e}}^{(i+1)})=(k-1)-n_{_e}^{(i)}+J_{_k}(n_{_e}^{(i)})+k \left \lfloor\dfrac{2n_{_e}^{(i)}-J_{_k}(n_{_e}^{(i)})-(k-2)}{k-1}\right \rfloor.
\end{equation}
Moreover, for any $n\in [[n_{_e}^{(i)}+1,n_{_{e}}^{(i+1)}]]$, we have
\begin{equation}\label{jk-n}
J_{_k}(n)=k(n-n_{_e}^{(i+1)})+J_{_k}(n_{_e}^{(i+1)}).
\end{equation}
\end{corollary}
\begin{proof}
Assume that $n_{_e}^{(i)}$ is a given high extremal point and its corresponding value through $J_{_k}$ is known. Then, by Theorem \ref{L 2-part2}(c), $J_{_k}(n_{_e}^{(i)}+1)=(k-1)-[n_{_e}^{(i)}-J_{_k}(n_{_e}^{(i)})]$.

Set $j:=(k-1)-[n_{_e}^{(i)}-J_{_k}(n_{_e}^{(i)})].$ Notice that $j-1=k-2-n_{_e}^{(i)}+J_{_k}(n_{_e}^{(i)}).$ Find $r_{_i}$ and $c_{_i}$ by employing equations \eqref{ri} and \eqref{ci}, respectively.

Then, on the one hand, by Theorem \ref{L 2-part2}(b), the next high extremal point of $J_{_k}$, $n_{_e}^{(i+1)}$, is
    \begin{equation*}
n_{_e}^{(i+1)}=a_{_k}(n_{_e}^{(i)},r_{_i})-\delta_{_k}(r_{_i},j)  
      =\dfrac{k(n_{_e}^{(i)}+1)-(r_{_i}+1)}{k-1}-c_{_i},
\end{equation*}
which tell us that \eqref{nei+1} holds. Also, it follows from Theorem \ref{L 2-part2}(b) that:
\begin{align*}  
J_{_k}(n_{_e}^{(i+1)})&=j+k\left \lfloor \dfrac{n_{_e}^{(i)}-(j-1)}{k-1} \right \rfloor \\
      &=(k-1)-n_{_e}^{(i)}+J_{_k}(n_{_e}^{(i)})+k\left \lfloor \dfrac{2n_{_e}^{(i)}-J_{_k}(n_{_e}^{(i)})-(k-2)}{k-1} \right \rfloor,
\end{align*}
which verifies the validity of \eqref{jk-nei+1}. Now, if $n\in [[n_{_e}^{(i)}+1,n_{_e}^{(i+1)}]]$, the point $(n,J_{_k}(n))$ can be determined by finding the intersection of the vertical line at $(n,n)$ and the line with slope $k$ that passes through $(n_{_e}^{(i+1)},J_{_k}(n_{_e}^{(i+1)}))$, which gives us \eqref{jk-n}.
\end{proof}

We now include a similar recurrence formula for the low extremal points. 

\begin{corollary}[Recurrence formula for computing low extremal points]\label{L 5} Let $\check n_{_{e}}^{(i)}\ge 2k-3$ be a low extremal point of the Josephus function $J_{_k}$, with functional value $J_{_k}(\check n_{_e}^{(i)})$.
Given $\check r_{_{i}}:={\rm mod}(\check n_{_e}^{(i)},k-1)$ and 
\begin{align*}
 \check c_{_{i}} &:= \begin{cases} 1, & \text{ if }\, \check  r_{_{i}} < J_{_k}(\check n_{_e}^{(i)})\\
0, & \text{ if }\, \check r_{_{i}} \geq J_{_k}(\check n_{_e}^{(i)}).\end{cases}
\end{align*}
Then,  we can compute the next low extremal point as follows:
\begin{equation}\label{check-ne}
\check n_{_{e}}^{(i+1)} =  \dfrac{k\check n_{_e}^{(i)}-\check r_{_i}}{k-1}-\check c_{_i}+1,
\end{equation} 
and its functional value as:
\begin{equation} \label{check-Jne}
J_{_k}(\check n_{_{e}}^{(i+1)})=k-\check n_{_e}^{(i+1)}+J_{_k}(\check n_{_e}^{(i)})+k \left \lfloor\dfrac{\check n_{_e}^{(i)}-J_{_k}(\check n_{_e}^{(i)})}{k-1}\right \rfloor.
\end{equation}
Moreover, for any $n\in [[\check n_{_e}^{(i)}+1,\check n_{_{e}}^{(i+1)}]]$, we have
\begin{equation}\label{Jkn-by-check}
J_{_k}(n)= kn + J_{_k}(\check n_{_e}^{(i+1)}) - (k-1)\check n_{_e}^{(i+1)}.\end{equation}
\end{corollary}
\begin{proof}
The proof of the formulas for $\check n_{_{e}}^{(i+1)}$ and $J_{_k}(\check n_{_{e}}^{(i+1)})$ are similar to Corollary \ref{L 3}. The expression for $J_{_k}(n)$ follows directly from \eqref{jk-n} and \eqref{jk-ne++1-new}.
\end{proof}
\begin{remark}[Knuth's formula is recovered]
When $k=2$, the above corollary can be used to derive Knuth's formula for the low and high extremal (fixed) points, as well as the explicit formula for $J_{_2}(n)$ for every $n\ge 1$. To do this, we substitute $k=2$ into equations \eqref{check-ne} and \eqref{check-Jne}, which yields $\check r_{{_i}}=0$ and $\check c_{{_i}}=1$ for all $i\ge 1$. More than that, Corollary \ref{coro2.4} can be used to start a straightforward induction argument, to show that the low extremal point $\check n_{{_e}}^{(i)}=2^{i}$ and the high extremal point, which in this case coincides with a fixed point, $n_{_p}^{(i)}=\check n_{_e}^{(i)}-1=2^{i}-1$ for all $i\ge 1$. Moreover, the explicit formula for $J_{_2}(n)$ can be obtained directly from \eqref{Jkn-by-check} by observing that $J_{_2}(\check n_{_e}^{(i)})=1$ and $i=\lfloor\log_2 n\rfloor$ whenever $n_{_e}^{(i)}$ is a low extremal point and $n\ge n_{_e}^{(i)}$. Specifically, we have $J_{_2}(n)= 2n-2^{\lfloor \log_2n \rfloor+1}+1$.
\end{remark}

We are now ready to prove the existence of infinitely many fixed points for $J_{_k}$ when $k\geq 3$, which has already been established by Knuth for $k=2$. The approach we take to prove the existence of fixed points will involve applying the recurrence form of Corollary \ref{L 3} successively.  We state this formally in the following theorem.

\begin{theorem}[Existence of infinitely many fixed points]
The Josephus function $J_{_k}$ will eventually reach a fixed point.
\end{theorem}
\begin{proof}
When $k=2$ the result was proved by Knuth \cite[pp. 162 \& 184]{Knuth}. So, let us assume that $k\ge 3$. Then, if $n \geq 2k-3$ and $J_{_k}(n)$ is known, and if $n$ is either a low extremal point or a point that is not a high extremal point, then we can reach a high extremal point by using Lemma \ref{L1}. 
Without loss of generality, we can assume that $n$ is a high extremal point that is not a fixed point of $J_{_k}$, and we seek to find a fixed point.

The proof will proceed by contradiction. Assume that the repeated application of Corollary \ref{L 3} generates only high extremal points, which are no fixed points of $J_{_k}$. Let us denote $n$ as $n_{_e}^{(0)}$. So, $n_{_e}^{(0)}+1$ is a low extremal point of $J_{_k}$, and by Proposition \ref{P1}, $j_{_0}:=J_{_k}(n_{_e}^{(0)}+1)<k-1$. Let $r_{_0}\leq  k-2$ such that $n_{_e}^{(0)}\equiv r_{_0} \mod{(k-1)}$. Then, by applying Corollary \ref{L 3}, the integer $n_{_e}^{(1)}$ is a high extremal point. Based on our general assumption  $r_{_0}\neq j_{_0}-1$. Notice that $n_{_e}^{(1)}$ can be written as follows:
\begin{equation*}
n_{_e}^{(1)}=\alpha(n_{_e}^{(0)}+1)-\dfrac{r_{_0}+1}{k-1}-c_{_0} \notag 
      =\alpha(n+1)+d_{_0}-1, \label{eqq 1}
\end{equation*}
where $\alpha := k/(k-1)$ and $d_{_0}:=(1-c_{_0})-(r_{_0}+1)/(k-1)$. Based on our assumption and on Corollary \ref{L 3}, an infinite sequence $\{n_{_e}^{(m)}\}_{m\in\mathbb{N}}$ of high extremal points can be generated, where none of its terms is a fixed point such that  $j_{_m}:=J_{_k}(n_{_e}^{(m)}+1)<k-1$ (a low extremal point), $n_{_e}^{(m)} \equiv r_{_m} \mod{(k-1)}$ ($r_{_m}\neq j_{_m}-1$), and for $m\geq 1$:
\begin{align}
n_{_e}^{(m)}&=\alpha^m(n_{_e}^{(0)}+1)+(\alpha^{m-1}d_{_0}+ \cdots + d_{_{m-1}})-1 \notag \\
&= \alpha^m(n+1)+\left ({\displaystyle \sum_{i=0}^{m-1}\alpha^{m-1-i}d_i} \right )-1,\label{eqq 2}
\end{align}
where $d_{_i}:=(1-c_{_i})-(r_{_i}+1)/(k-1)$.
Therefore, \eqref{eqq 2} holds for all $m\geq 1$. Now, observe that it can be rewritten as follows:
\begin{align}  \nonumber
n_{_e}^{(m)}&= \alpha^{m-1}\left(\alpha(n+1)+ \sum_{i=0}^{m-1}\dfrac{d_{_i}}{\alpha^i}-\frac{1}{\alpha^{m-1}}\right)\\ \nonumber
&= \frac{k^{m-1}}{(k-1)^{m-1}}\left(\dfrac{k}{k-1}(n+1)+ \dfrac{1}{k-1}\sum_{i=0}^{m-1}\dfrac{(k-1)d_{_i}}{\alpha^i}-\frac{(k-1)^{m-1}}{k^{m-1}}\right) \\
&=\frac{k^{m-1}}{(k-1)^{m}}\left(k(n+1)+\sum_{i=0}^{m-1}(k-1)d_{_i}\left(\dfrac{k-1}{k}\right)^i-\frac{(k-1)^{m}}{k^{m-1}}\right)\label{s_m-eq1}.
\end{align} Note further that $(k-1)d_{_i}=(k-1)(1-c_{_i})-r_{_i}-1,$ where $r_{_i}\in [[0,k-2]]$ and $c_{_i}\in\{0,1\}$. Hence, if we define
$$\beta_m:=k(n+1)+\sum_{i=0}^{m-1}(k-1)d_{_i}\left(\dfrac{k-1}{k}\right)^i-\frac{(k-1)^{m}}{k^{m-1}},$$ then, for all $m\geq 1$, $\beta_m$ satisfies:
\begin{equation}\label{t_m-bounded}
1\le k(n-k+2)-1\le \beta_m \le k(n+k-1),
\end{equation} where we have used that $-(k-1)\le(k-1)d_{_i}\le k-2$, the fact that $n\ge 2k-3$ implies that $n\ge k-1$ because $k\ge 3$, and $$\sum_{i=0}^{m-1}\left(\dfrac{k-1}{k}\right)^i=\displaystyle\dfrac{1-\left(\dfrac{k-1}{k}\right)^{m}}{1-\left(\dfrac{k-1}{k}\right)}\le k.$$
Using now \eqref{s_m-eq1}, we get
\begin{equation*}
n_{_e}^{(m)}=\frac{k^{m-1}}{(k-1)^{m}}\,\cdot \beta_m,
\end{equation*}and hence, $n_{e}^{(m)}$ is a positive integer number for all $m$, if $\beta_m$ grows as $(k-1)^m$ for $m$ large, or it approaches $0$ as $\frac{(k-1)^{m}}{k^{m-1}}$. These two possibilities do not occur because \eqref{t_m-bounded} guarantees that $\beta_m$ is uniformly far from $0$ and bounded. Thus, for some $m_0 \geq 1$, $r_{m_0}$ coincides with $j_{m_0}-1$, and $n_{e}^{(m_0)}$ is a fixed point by Theorem \ref{L 2-part2}(c), proving the result.
\end{proof}

\section{Numerical Experiments}

In this section, we present a numerical comparison of four algorithms to evaluate the Josephus function. The computational experiments were carried out on an iMac 3.6 GHz 10-Core Intel Core i9 with 32GB of RAM. The algorithms were implemented in the Julia programming language v1.8.

\subsection{The algorithms}

First, let's describe three well-known methods for solving the Josephus problem (Euler, Knuth and Uchiyama algorithms); see, for instance, \cite{Gra, Uch,Woo}. Moreover, we propose a novel scheme for evaluating the Josephus function called the Extremal algorithm.

\subsubsection{Euler algorithm}

The Euler algorithm is a natural interpretation of the recurrence relation proposed first in \cite[\& 8, pp. 130-131]{Eul}. 
In \cite[pp. 57]{Woo}, Woodhouse gave the following modern algorithmic description:
\begin{equation*}
J_{_k}(n)=((\cdots ((1+k)_{(2)}+k)_{(3)}+\cdots +k)_{(n-1)}+k)_{(n)}
\end{equation*}
where for any positive integers $r$ and $s$ the symbol $(r)_{(s)}$ denotes the integer satisfying: $(r)_{(s)}\equiv r \mod{s}$ and  $1\leq (r)_{(s)} \leq s$. This scheme clearly has an intrinsic recursive behavior and complexity $O(n)$, \emph{i.e.}, there is necessary about $n$ function evaluations to reach the solution. 

To compute the Josephus function $J_{_k}$ at $n$, the Euler algorithm requires evaluating $J_{_k}$ for all preceding values, that is, $n-1$ evaluations of $J_{_k}$. 

\subsubsection{Knuth algorithm}

In equation (3.19) from \cite[Chapter 3, Section 3, pp. 81]{Gra} Knuth presents the following algorithm: 
\begin{align*}
& D := 1; \\
& \text{While } D\leq (k-1)n \text{ do } D := \left\lceil \dfrac{k}{k-1}D \right \rceil; \\ 
& J_{_k}(n) = kn+1-D. 
\end{align*}
Knuth algorithm significantly reduces the number of evaluations required to compute $J_{_{k}}(n)$ compared to the Euler algorithm. Instead of the recursive computation of the Euler algorithm, the Knuth algorithm only requires $O(\ln(n))$ evaluations before computing $J_{_k}(n)$. This results in a more efficient computation process than Euler's.

\subsubsection{Uchiyama algorithm}

In equation (11) from \cite[Section 4, pp. 329]{Uch}, Uchiyama presents the following algorithm:

Let $n_{_1}=1$, $$c_{_1}^{\star}=c_{_1}:=J_{_k}(2)=\begin{cases} 1, & \text{ if } k \text{ is even }\\
2, & \text{ if } k \text{ is odd }.\end{cases}$$ For $i\geq 1$ compute:
\begin{equation*}
\begin{aligned}
& n_{_{i+1}} =  \left\lfloor \dfrac{k(n_{_i}+1)-c_{_i}}{k-1} \right \rfloor;  \\
& c^{\star}_{_{i+1}}:=c_{_i}+(k-1)(n_{_{i+1}}+1)-k(n_{_i}+1); \\ 
& c_{_{i+1}} \equiv \begin{cases} c^{\star}_{_{i+1}} \mod {(n_{_{i+1}}+1)}, & \text{ if } 1\leq c_{_{i+1}}\leq n_{_{i+1}}+1\\ c^{\star}_{_{i+1}}, & \text{otherwise}.\end{cases}  
\end{aligned}
\end{equation*}
If $n_{_i}< n \leq n_{_{i+1}}$, $
J_{_k}(n)=c_{_i}+k(n-n_{_i}-1)$.

Uchiyama algorithm, like the Knuth algorithm, improves to $O(\ln(n))$ the number of evaluations required to compute $J_{_{k}}(n)$ compared to the Euler algorithm. 

\subsubsection{Extremal algorithm}

To evaluate the Josephus function $J_{_k}$, we proposed the Extremal algorithm, which employs a strategy of computing recursively high extremal points $n^{(i)}_{_e}$ $(i=1,2,\ldots,m)$ until $n^{(m)}_{_e}$ is greater than or equal to $n$; see Corollary \ref{L 3}. This approach capitalizes on the linear piecewise structure of the Josephus function, enabling the Extremal algorithm to efficiently compute $J_{_k}(n)$ compared to other methods.

The Extremal algorithm for evaluating $J_{_k}$ at $n\ge 2k-3$ can be described as follows. We start by defining $j:=J_{_k}(2k-2)$ and our first high extremal point $n_{_e}^{(1)}$ and $J_{_k}(n_{_e}^{(1)})$ as follows: 
$$\left(n_{_e}^{(1)},J_{_k}(n_{_e}^{(1)})\right):=\begin{cases} (2k-3,j+k-2), & \text{ if  } j \leq k-1\\
(2k-2,j), & \text{ if  } j > k-1.\end{cases}$$ 
\begin{equation*}
\left.\begin{aligned}
& r_{_{i}}:={\rm mod}(n_{_e}^{(i)},k-1); \\ 
& c_{_{i}} := \begin{cases} 1, & \text{ if } r_{_{i}} < k-2-n_{_e}^{(i)}+J_{_k}(n_{_e}^{(i)})\\
0, & \text{ if } r_{_{i}} \geq k-2-n_{_e}^{(i)}+J_{_k}(n_{_e}^{(i)});\end{cases}
\\ 
& J_{_k}(n_{_{e}}^{(i+1)})=(k-1)-n_{_e}^{(i)}+J_{_k}(n_{_e}^{(i)})+k \left \lfloor\dfrac{2n_{_e}^{(i)}-J_{_k}(n_{_e}^{(i)})-(k-2)}{k-1}\right \rfloor; \\ 
& n_{_{e}}^{(i+1)} =  \dfrac{k(n_{_e}^{(i)}+1)-(r_{_i}+1)}{k-1}-c_{_i}. 
\end{aligned}\right \}
\end{equation*}
If $n_{_e}^{(i)}< n \leq n_{_e}^{(i+1)}$, then $
J_{_k}(n)=J_{_k}(n_{_e}^{(i+1)})+k(n-n_{_e}^{(i+1)}).
$



The Extremal algorithm is also more efficient than Euler algorithm. Like the Knuth and Uchiyama algorithms, it requires $O(\ln(n))$ evaluations before computing $J_{_k}(n)$. 
\subsection{A Comparative Computational Study} 
To compare the performance of all algorithms, we conducted extensive numerical experiments by measuring their CPU time. We employed a Performance Profile testing a total of $452500$ problems, uniformly distributed for $n \in [50000:100:100000]$ and $k \in [50:10:1000]$. We compared the Extremal algorithm with three other algorithms, namely Euler, Knuth, and Uchiyama, and our results indicate that the Extremal algorithm outperformed the other three in almost all problem instances (see Figure \ref{f2}), making it an excellent choice for solving the Josephus problem for large inputs.
\begin{figure}[!h]
    \centering
    \begin{subfigure}[b]{0.49\textwidth}
    \includegraphics[width=1\textwidth]{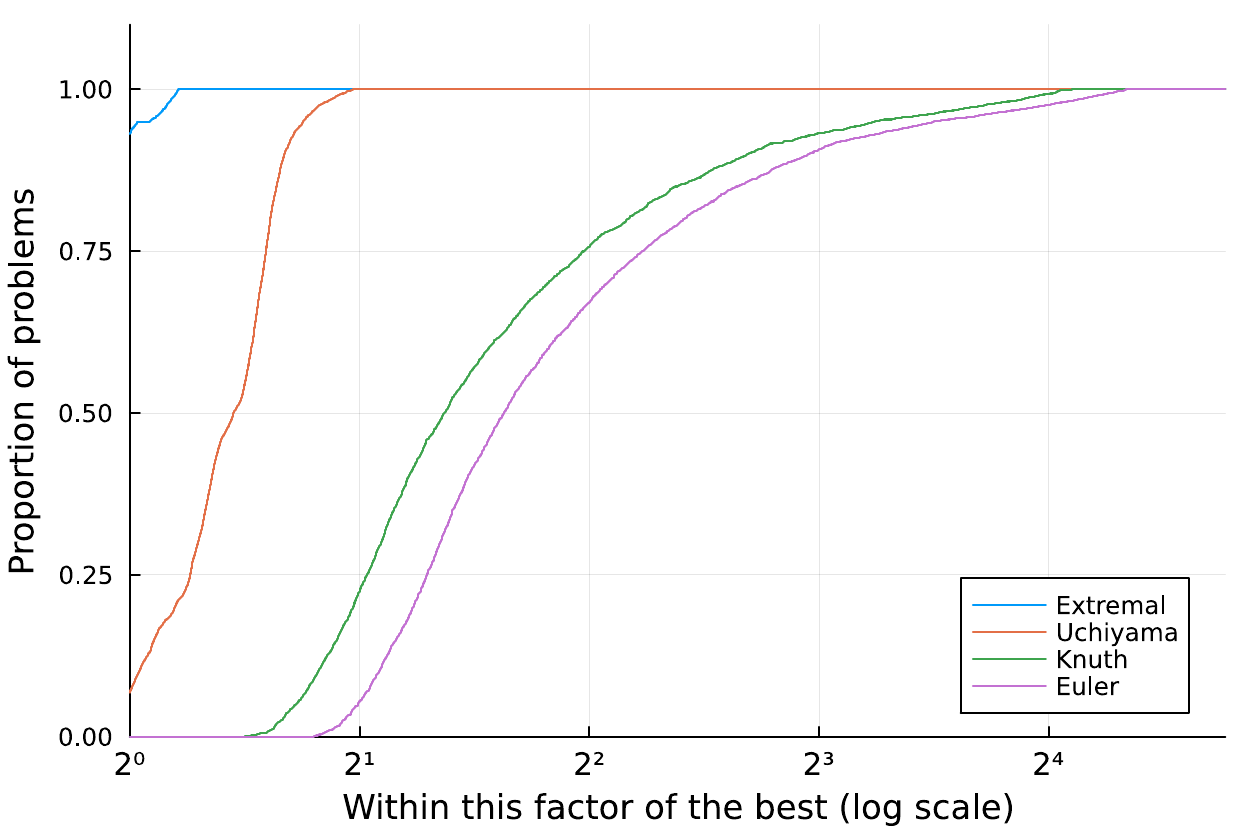}
    \caption{PP: All Methods}
    \label{fig:2.1}
    \end{subfigure}
    \hfill
    \begin{subfigure}[b]{0.49\textwidth}
    \includegraphics[width=1\textwidth]{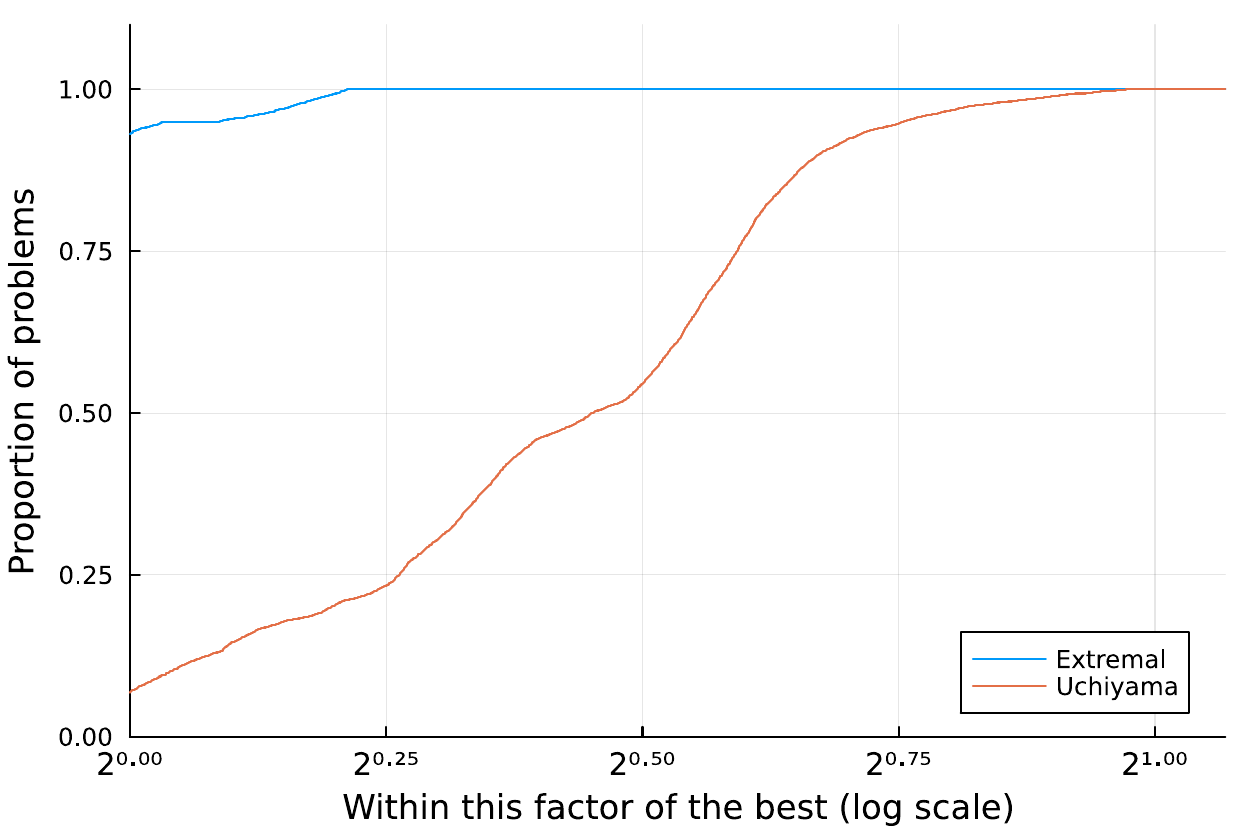}
    \caption{PP: Extremal vs Uchiyama}
    \label{fig:2.2}
    \end{subfigure}
    \hfill
    \caption{Performance Profiles (PP) for $n\in[50000:100:100000], k\in[50:10:1000]$}
    \label{f2}
    \end{figure}

To provide a more detailed analysis of our results, we present Table \ref{t2}, which summarizes descriptive statistics of the benchmark of problems, including the minimum (min), mean,  maximum (max), and standard deviation (std) of total CPU time. The Extremal algorithm consistently performed better than Euler, Knuth, and Uchiyama algorithms, indicating its superiority for solving the Josephus problem. 
\begin{table}[!h]
  \caption{\label{tab:PolySOC_tau=0.25} Statistics of the experiments} \label{t2}
  \centering 
  \begin{tabular}{l|l | l | l | l  }
{\bfseries Algorithms} \quad \quad  & 
{\bfseries min} \quad  &{\bfseries mean} &   
    {\bfseries max} \quad & {\bfseries std} \quad  \\
    \hline
    Euler  & 0.000813209 & 0.00120984 & 0.00161632& 0.000231202\\ 
    Knuth  & 0.000659197 & 0.00101147 & 0.00137245& 0.00019701\\ 
    Uchiyama  & 7.234e-5 & 0.000518487 & 0.0010305& 0.000246989\\ 
    Extremal  & 7.6873e-5 & 0.000367613 & 0.000633953& 0.000144699\\ 
\end{tabular}
  \end{table}

\section{Concluding Remarks}
In this paper, we presented a novel study characterizing the Josephus function structure. We use the function's piecewise linear structure to identify extremal points (including fixed points) of the Josephus function, $J_{_k}$, via a recurrence formula. We have developed an efficient algorithm for evaluating $J_{_k}$ for large values of $n$ based on the successive computation of the high extremal points of $J_{_k}$. The effectiveness of the proposed scheme was validated through its comparison to established algorithms. The results of the comparative study demonstrate the remarkable performance of the proposed approach in computing the Josephus function for large inputs. It is noteworthy that we can employ the recurrence formula to calculate low extremal points, and a similar approach can be designed to address the Josephus problem. The analytical study presented in this paper can have substantial practical implications for applications such as scheduling, network optimization, and distributed algorithms. However, finding a recurrence formula for computing fixed points, analogous to the one we obtained for extremal points, remains an open problem. Addressing this problem may lead to further insights and improvements in solving the Josephus problem.

\bigskip
\hrule
\bigskip
\noindent 2010 {\it Mathematics Subject Classification}: Primary 65Q30;
 Secondary 11Y55, 11B50.

\noindent \emph{Keywords:}  Josephus Function, Fixed Points, Extremal Points, Piece-wise Linear Structure.

\bigskip
\hrule
\bigskip

\noindent (Concerned with sequences
\seqnum{A000225}
and
\seqnum{A182459}.)

\bigskip
\hrule
\bigskip

\vspace*{+.1in}
\noindent

\bigskip
\hrule
\bigskip

\noindent

\vskip .1in

\end{document}